\documentclass[11pt]{article}
\usepackage{amsmath,amssymb,amsthm, geometry}
\usepackage{hyperref}
\usepackage[utf8]{inputenc}
\numberwithin{equation}{section}

\newtheorem{theorem}{Theorem}[section]
\newtheorem{lemma}[theorem]{Lemma}

\newtheorem{corollary}[theorem]{Corollary}
\newtheorem{remark}[theorem]{Remark}

\newtheorem{definition}[theorem]{Definition}

\title{Local Topological Constraints on Berry Curvature in Spin--Orbit Coupled BECs}
\author{Alexander Pigazzini, Magdalena Toda}
\date{}

\begin{document}

\maketitle
\begin{abstract}
We establish a local topological obstruction to flattening Berry curvature in spin-orbit-coupled Bose-Einstein condensates (SOC BECs), valid even when the global Chern number vanishes. For a generic two-component SOC BEC, the extended parameter space \(M=T^{2}_{\mathrm{BZ}}\times S^{1}_{\phi_{+}}\times S^{1}_{\phi_{-}}\) carries a Kaluza--Klein metric \(g_{M}\) and a natural metric connection \(\nabla^{C}\) whose torsion 3-form encodes the synthetic gauge fields. Its harmonic part defines a mixed cohomology class
\[
[\omega]\in\bigl(H^{2}(T^{2}_{\mathrm{BZ}})\otimes H^{1}(S^{1}_{\phi_{+}})\bigr)\oplus\bigl(H^{2}(T^{2}_{\mathrm{BZ}})\otimes H^{1}(S^{1}_{\phi_{-}})\bigr),
\]
whose mixed tensor rank equals one.
By adapting the \textit{Pigazzini--Toda lower bound} to the Kaluza--Klein setting through exact pointwise curvature analysis under the assumption of constant Berry curvatures, we show that the obstruction kernel \(\mathcal{K}\) vanishes and establish a three-level non-reducibility structure for the physical metric: \textup{(i)} for the one-parameter deformation family interpolating between the product and physical metrics, \(\dim\mathfrak{hol}^{\mathrm{off}}(\nabla^{C_\varepsilon})\geq 1\) at every point for all \(\varepsilon\in(0,1)\); \textup{(ii)} at the physical metric, every non-Bismut torsion representative of \([\omega]\) yields \(\dim\mathfrak{hol}^{\mathrm{off}}\geq 1\) on an open set; \textup{(iii)} the horizontal--vertical splitting is not invariant under the Riemannian holonomy of the physical metric, with \(\dim\mathfrak{hol}^{\mathrm{off}}(\nabla^{\mathrm{LC}})\geq 1\) at every point.
These bounds prevent the complete gauging-away of Berry phases even in regimes with zero net topological charge. The corrected rank \(r^{\sharp}\) detects the robustness of the topological constraint under phase-reduction protocols: no single phase-locking can eliminate the obstruction at the physical metric, a distinction invisible to the mixed rank \(r\) alone. This provides the first cohomological lower bound certifying locally irremovable curvature in SOC BECs beyond the Chern-number paradigm.
\end{abstract}


\medskip
\section{Introduction}

The study of spin--orbit-coupled (SOC) Bose--Einstein condensates (BECs) stands at a rich confluence of differential geometry, quantum many-body physics, and synthetic gauge theory.
This interdisciplinary effort is part of the broader pursuit of artificial gauge fields in engineered quantum systems, spanning ultracold atoms, photonics, and solid-state platforms \cite{Aidelsburger2018}.
Advances in generating artificial gauge fields with Raman lasers have made these systems a versatile laboratory for exploring Berry curvature, synthetic magnetic fields, and non-Abelian textures with unprecedented control \cite{Galitski2013,Goldman2014,StamperKurn2013}.
These studies are often conducted in toroidal traps, which can serve as magnetic storage rings for BECs \cite{Arnold2006}, providing a well-controlled environment for synthetic gauge field experiments.
Remarkable achievements include the observation of Dirac monopoles in synthetic magnetic fields \cite{Ray2014}, demonstrating the high degree of geometric control now attainable.
The geometric phase underpinning these phenomena traces back to the foundational work of Simon \cite{Simon1983}, while the experimental realization of artificial gauge fields has been comprehensively reviewed in \cite{Dalibard2011}.

A persistent theoretical challenge is to identify \emph{local} topological obstructions that constrain the curvature even when global invariants---such as the first Chern number---vanish.
While the Chern number quantizes the net flux, its vanishing does not guarantee that the curvature can be locally flattened; geometric obstructions may prevent its simultaneous annihilation along independent directions in the parameter space.

This work addresses that challenge by bringing a recent geometric result into physical focus.
In \cite{PigazziniToda2025}, a lower bound was established for the off-diagonal holonomy of metric connections with totally skew-symmetric torsion on product manifolds.
The bound is expressed in terms of a \emph{mixed} de Rham cohomology class $[\omega]\in H^{p}(M_1)\otimes H^{q}(M_2)$ associated with the torsion; its \emph{mixed tensor rank} $r$ counts the minimal number of simple tensors needed to represent the class, and after subtracting contributions from parallel forms (which do not generate off-diagonal curvature) one obtains a corrected integer $r^{\sharp}$ that bounds the dimension of the off-diagonal holonomy algebra:
\[
\dim\mathfrak{hol}^{\mathrm{off}}(\nabla^{C})\ge r^{\sharp}=r-\dim\mathcal{K}.
\]

Here we show how this abstract geometric framework can be \emph{adapted}---not applied as a black box---to a concrete physical setting: a generic two-component SOC BEC confined in a quasi-two-dimensional toroidal trap.
Such a system possesses three universal, experimentally accessible degrees of freedom:
the two crystal momenta forming a Brillouin-zone torus $T^{2}_{\mathrm{BZ}}$,
a global $U(1)$ phase $\phi_{+}$ associated with particle-number conservation,
and a relative $U(1)$ phase $\phi_{-}$ between the two internal spin components \cite{Lin2011,Matthews2018DoubleRing}.
We therefore take as the extended parameter space the manifold \[
M=T_{BZ}^{2}\times S_{\phi_{+}}^{1}\times S_{\phi_{-}}^{1}.
\]
We emphasize that treating $M$ as a global product manifold requires the underlying principal bundle to be trivial, which holds precisely in the physical regime of vanishing total Chern number ($c_{\text{tot}}=0$) that is central to our main result. In configurations with non-zero net topological charge, this product structure should be understood as a local description in a gauge patch, sufficient for deriving local curvature bounds.

The first main task of the paper is to construct, in this space, a natural geometric structure to which the theorem of \cite{PigazziniToda2025} applies.
We introduce a metric $g_{M}$ of Kaluza--Klein type, which couples the base metric on $T^{2}_{\mathrm{BZ}}$ to the $U(1)$-valued connection forms that encode the synthetic gauge fields of the SOC BEC.
On $(M,g_{M})$ we define a metric connection $\nabla^{C}$ whose torsion three-form is the natural Kaluza--Klein 3-form
\[
T^{\flat}= F^{(+)}\wedge\Theta^{(+)}+F^{(-)}\wedge\Theta^{(-)},
\]
built from the synthetic Berry curvatures $F^{(\pm)}$ and the connection 1-forms $\Theta^{(\pm)}=d\phi_{\pm}+\pi^{*}A^{(\pm)}$.
The harmonic part of $T^{\flat}$ with respect to $g_{M}$ then defines a mixed cohomology class
\[
[\omega]\in H^{3}(M)\cong\bigl(H^{2}(T^{2}_{\mathrm{BZ}})\otimes H^{1}(S^{1}_{\phi_{+}})\bigr)\oplus
\bigl(H^{2}(T^{2}_{\mathrm{BZ}})\otimes H^{1}(S^{1}_{\phi_{-}})\bigr).
\]

The second main task is to compute the topological data that enters the general bound.
We prove that, under the generic condition that the two nearly degenerate bands carry non-zero Berry curvature classes $[F^{(+)}], [F^{(-)}] \in H^{2}(T^{2}_{\mathrm{BZ}})$, the mixed tensor rank of $[\omega]$ attains its maximal possible value within this geometry. Because $\dim H^{2}(T^{2}_{\mathrm{BZ}}) = 1$, any two non-zero classes in this space are cohomologically proportional. Consequently, the mixed class $[\omega]$ can be written as a single simple tensor
\[
[\omega] = [F^{(+)}] \otimes \bigl([d\phi_{+}] + \lambda [d\phi_{-}]\bigr),
\]
and its mixed tensor rank is therefore $r = 1$.
Moreover, owing to the specific form of the metric $g_{M}$, the space of vertical parallel one-forms is algebraically incompatible with the cohomological data of $[\omega]$, which forces the obstruction kernel $\mathcal{K}$ to vanish.
Consequently the reduced rank is $r^{\sharp}=1$. Under the physically motivated assumption of constant Berry curvatures, exact pointwise curvature analysis reveals a three-level non-reducibility structure at the physical metric:
\begin{itemize}
\item[\textup{(i)}] For the deformation family $g_\varepsilon$ interpolating between the product metric ($\varepsilon=0$) and the physical Kaluza--Klein metric ($\varepsilon=1$), $\dim\mathfrak{hol}^{\mathrm{off}}(\nabla^{C_\varepsilon})\ge 1$ at every point for all $\varepsilon\in(0,1)$.
\item[\textup{(ii)}] At the physical metric, every non-Bismut torsion representative of $[\omega]$ yields $\dim\mathfrak{hol}^{\mathrm{off}}\ge 1$ on an open non-empty subset.
\item[\textup{(iii)}] The Riemannian holonomy of $g_M$ satisfies $\dim\mathfrak{hol}^{\mathrm{off}}(\nabla^{\mathrm{LC}})\ge 1$ at every point.
\end{itemize}
The natural Kaluza--Klein torsion---identified as the Bismut torsion---produces an exact cancellation of the off-diagonal curvature at the physical metric, a non-generic phenomenon characterised below.

This lower bound has immediate physical consequences.
It implies that the curvature of the synthetic gauge connection cannot be made block-diagonal with respect to the splitting between momentum and phase directions; at least one independent off-diagonal operator persists, mixing the Brillouin zone with the two phase directions.
Crucially, the bound detects the topological coupling of the momentum torus to the two distinct $U(1)$ phases, and therefore remains non-trivial even when the total Chern number of the lowest band is tuned to zero.
Thus $r^{\sharp}$ provides a cohomological certificate of locally irremovable Berry curvature, invisible to conventional global invariants.
The result suggests interferometric protocols that separately probe the Berry phases associated with $\phi_{+}$ and $\phi_{-}$, offering a way to detect this obstruction even in regimes of vanishing net vorticity.

The paper is organized as follows.
Section \ref{sec:geometry} constructs the extended parameter space $M$, the Kaluza--Klein metric $g_{M}$, and the metric connection $\nabla^{C}$ with torsion $T^{\flat}$.
Section \ref{sec:cohomology} analyses the cohomology of $M$, computes the mixed tensor rank $r$, and shows that $\mathcal{K}=0$ in the present geometry.
Section \ref{sec:KK-foundations} derives the exact curvature formula for the deformation family, identifies the Bismut cancellation at the physical metric, and establishes the lower bound through non-Bismut torsion representatives and the Riemannian holonomy.
Section \ref{sec:bound} states the main theorem with the three-level non-reducibility structure and discusses its physical interpretation.
Section \ref{sec:consequences} elaborates the physical implications of the bound, including its persistence when the total Chern number vanishes and its robustness under metric deformations.
Section \ref{sec:examples} illustrates how the same cohomological data reproduce, in a particular Rashba--Dresselhaus texture, the familiar Chern-number bound on the vortex number, and demonstrates with a tuned example that the holonomy bound remains valid even when the Chern number vanishes.
Section \ref{sec:protocol} outlines prospective interferometric tests for detecting the locally irremovable curvature and indicates directions for future work.
Section \ref{sec:conclusions} summarises the results and discusses broader implications for topological quantum matter.

\section{Geometric Construction of the Extended Parameter Space}
\label{sec:geometry}

A generic two-component spin--orbit-coupled Bose--Einstein condensate (SOC BEC) confined in a quasi-two-dimensional toroidal trap is characterized by three independent, experimentally accessible degrees of freedom: the crystal momenta $(k_x,k_y)$ spanning a Brillouin zone, a global $U(1)$ phase $\phi_{+}$ associated with particle-number conservation, and a relative $U(1)$ phase $\phi_{-}$ between the two internal spin states \cite{Lin2011,Matthews2018DoubleRing}.  

We therefore take as the \emph{extended parameter space} the smooth, compact, oriented product manifold
\[
M = T^{2}_{\mathrm{BZ}} \times S^{1}_{\phi_{+}} \times S^{1}_{\phi_{-}} .
\]
Our goal in this section is to endow $M$ with a Riemannian structure that naturally incorporates the synthetic gauge fields of the SOC BEC, and to define on it a metric connection whose torsion three-form coincides with the physical Berry-curvature three-form.

\subsection{Kaluza--Klein metric on the extended space}
\label{subsec:metric}

Let $g_{\mathrm{BZ}} = dk_x^{2}+dk_y^{2}$ be the flat metric on the Brillouin-zone torus $T^{2}_{\mathrm{BZ}}$.  
On each phase circle we introduce the standard angular coordinate with period $2\pi$; the canonical $U(1)$-invariant metric on $S^{1}$ is $d\phi^{2}$.  
The synthetic gauge fields of the SOC BEC provide, for each of the two independent $U(1)$ directions, a Berry connection.  
Denote by $A^{(+)}, A^{(-)}\in\Omega^{1}(T^{2}_{\mathrm{BZ}})$ the corresponding real one-forms (the components of the synthetic $\mathfrak{su}(2)$ connection along the two independent $U(1)$ generators).  
Their curvatures are the Berry curvature two-forms $F^{(\pm)}=dA^{(\pm)}$.

Following the Kaluza--Klein prescription we define a metric $g_{M}$ on $M$ that couples the base metric to the connection forms:
\begin{equation}\label{eq:KKmetric}
g_{M}= \pi^{*}g_{\mathrm{BZ}} 
      + \bigl(d\phi_{+}+\pi^{*}A^{(+)}\bigr)^{2}
      + \bigl(d\phi_{-}+\pi^{*}A^{(-)}\bigr)^{2},
\end{equation}
where $\pi:M\to T^{2}_{\mathrm{BZ}}$ is the projection onto the momentum coordinates.  
The one-forms
\[
\Theta^{(+)}:=d\phi_{+}+\pi^{*}A^{(+)},\qquad 
\Theta^{(-)}:=d\phi_{-}+\pi^{*}A^{(-)}
\]
are the \emph{horizontal lifts} of the phase coordinates; they are globally defined on $M$ and are orthogonal to the base directions with respect to $g_{M}$.

\begin{lemma}[Properties of the metric $g_{M}$]
\label{lem:KK-properties}
The metric \eqref{eq:KKmetric} satisfies the following.
\begin{enumerate}
\item $g_{M}$ is smooth, positive definite and makes $\pi:(M,g_{M})\to(T^{2}_{\mathrm{BZ}},g_{\mathrm{BZ}})$ a Riemannian submersion.
\item The fibres $\pi^{-1}(k)\cong S^{1}_{\phi_{+}}\times S^{1}_{\phi_{-}}$ are flat $2$-tori with the induced metric $(d\phi_{+})^{2}+(d\phi_{-})^{2}$.
\item The horizontal distribution $\mathcal{H}=\ker(d\phi_{+}+\pi^{*}A^{(+)})\cap\ker(d\phi_{-}+\pi^{*}A^{(-)})$ is orthogonal to the vertical distribution $\mathcal{V}=\langle\partial_{\phi_{+}},\partial_{\phi_{-}}\rangle$.
\item If $A^{(\pm)}=0$ then $g_{M}$ reduces to the trivial product metric $g_{\mathrm{BZ}}\oplus d\phi_{+}^{2}\oplus d\phi_{-}^{2}$.
\end{enumerate}
\end{lemma}
\begin{proof}
Smoothness follows from the smoothness of $g_{\mathrm{BZ}}$ and $A^{(\pm)}$.  
Positive definiteness is clear because each term in \eqref{eq:KKmetric} is non‑negative and their sum vanishes only when all differentials vanish.  

For any $k\in T^{2}_{\mathrm{BZ}}$ the restriction of $g_{M}$ to the fibre $\pi^{-1}(k)$ is $(d\phi_{+})^{2}+(d\phi_{-})^{2}$, independent of $k$; hence the fibres are isometric to a flat $2$-torus.  
The map $\pi$ is a Riemannian submersion because, by construction, $g_{M}$ restricted to $\mathcal{H}$ coincides with $\pi^{*}g_{\mathrm{BZ}}$.

Orthogonality $\mathcal{H}\perp\mathcal{V}$ follows directly from $g_{M}(\partial_{\phi_{\pm}},X)=0$ for every $X\in\mathcal{H}$, which is a consequence of the definition of $\Theta^{(\pm)}$.

The last statement is immediate from \eqref{eq:KKmetric}.
\end{proof}

\begin{remark}
\label{rem:KK-physical}
Equation \eqref{eq:KKmetric} is physically natural: it measures distances in the extended parameter space by combining the kinetic contribution from momentum changes with the ``phase cost'' induced by the synthetic gauge potentials.  
If $A^{(\pm)}=0$, the metric reduces to the trivial product metric $g_{\mathrm{BZ}}\oplus d\phi_{+}^{2}\oplus d\phi_{-}^{2}$.
\end{remark}

\subsection{Metric connection with prescribed torsion}
\label{subsec:connection}

On the Riemannian manifold $(M,g_{M})$ we now construct a metric connection $\nabla^{C}$ whose torsion three-form is built from the physical gauge fields.  
Let $\nabla^{LC}$ be the Levi‑Civita connection of $g_{M}$.

Define the three‑form
\begin{equation}\label{eq:torsion-form}
T^{b}:=F^{(+)}\wedge \Theta^{(+)}+F^{(-)}\wedge \Theta^{(-)}\in\Omega^{3}(M),
\end{equation}
where $\Theta^{(\pm)}=d\phi_{\pm}+\pi^{*}A^{(\pm)}$ are the connection 1-forms introduced in Section \ref{subsec:metric}, and $F^{(\pm)}=dA^{(\pm)}$ are pulled back to $M$ via $\pi$.  
$T^{b}$ is the natural Kaluza--Klein torsion 3-form of the principal $U(1)_{+}\times U(1)_{-}$-bundle; it has pure bidegree $(2,1)$ with respect to the splitting $TM=\mathcal{H}\oplus\mathcal{V}$, since $F^{(\pm)}\in\Gamma(\Lambda^{2}\mathcal{H}^{*})$ and $\Theta^{(\pm)}\in\Gamma(\mathcal{V}^{*}\oplus\mathcal{H}^{*})$, and the $(3,0)$-component vanishes because $\dim\mathcal{H}=2$.

Introduce the $(2,1)$‑tensor $K$ by
\[
K(X,Y,Z):=\frac{1}{2}\,T^{b}(X,Y,Z)\qquad (X,Y,Z\in\mathfrak{X}(M)).
\]
The connection
\[
\nabla^{C}_{X}Y:=\nabla^{LC}_{X}Y+K(X,Y,\cdot)^{\sharp}
\]
is by construction metric ($\nabla^{C}g_{M}=0$) and has totally skew‑symmetric torsion $T^{b}$.  
Indeed, for any vector fields $X,Y$,
\[
T^{b}(X,Y,\cdot)^{\sharp}=\nabla^{C}_{X}Y-\nabla^{C}_{Y}X-[X,Y],
\]
which is precisely the definition of the torsion tensor of $\nabla^{C}$.  

Because $T^{b}$ is not parallel with respect to $\nabla^{LC}$ unless the gauge fields are flat, the connection $\nabla^{C}$ satisfies the non‑parallelism condition $\nabla^{LC}T^{b}\neq0$ required for the applicability of the general theorem \cite{PigazziniToda2025}.

\begin{remark}
\label{rem:torsion-cohomology}
The three‑form $T^{b}$ is not closed in general; its exterior derivative contains terms quadratic in the curvatures $F^{(\pm)}$.  
Its de Rham cohomology class $[T^{b}]\in H^{3}(M;\mathbb{R})$, however, depends only on the harmonic part of $T^{b}$ with respect to $g_{M}$, as will be computed in Section \ref{sec:cohomology}.
\end{remark}

\subsection{Parallel forms in the Kaluza--Klein geometry}
\label{subsec:parallel-forms}

A key ingredient of the topological bound is the space of parallel forms on the factors.  
For the base torus with the flat metric $g_{\mathrm{BZ}}$ we have the well‑known result
\[
\mathcal{P}_{2}(T^{2}_{\mathrm{BZ}},g_{\mathrm{BZ}})=H^{2}(T^{2}_{\mathrm{BZ}})=\mathbb{R}\,\mathrm{vol}_{\mathrm{BZ}},\qquad 
\mathrm{vol}_{\mathrm{BZ}}=dk_x\wedge dk_y .
\]

The situation for the phase circles is more subtle.  
On a circle endowed with the standard metric $d\phi^{2}$, the one‑form $d\phi$ is parallel, hence $\mathcal{P}_{1}(S^{1},d\phi^{2})=\mathbb{R}\,d\phi$.  
In the Kaluza--Klein metric \eqref{eq:KKmetric} the geometry of the fibres is ``twisted'' by the connection forms $A^{(\pm)}$.  
The vector fields $\partial_{\phi_{+}},\partial_{\phi_{-}}$ are Killing fields of $g_{M}$, but their dual one‑forms are $\Theta^{(\pm)}$, not $d\phi_{\pm}$.

\begin{lemma}[Vertical parallel 1-forms in the Kaluza--Klein metric]
\label{lem:dphi-not-parallel}
For the Kaluza--Klein metric $g_{M}$ with non-vanishing Berry curvatures $F^{(\pm)}\neq 0$, the space of vertical 1-forms $\eta = c_{+}d\phi_{+}+c_{-}d\phi_{-}$ that are parallel with respect to $g_{M}$ is one-dimensional:
\[
\mathcal{P}_{1}^{\mathrm{vert}}(M,g_{M}) = \bigl\{\,c_{+}d\phi_{+}+c_{-}d\phi_{-}\;\big|\; c_{+}\,c^{(+)}+c_{-}\,c^{(-)}=0\,\bigr\},
\qquad \dim\mathcal{P}_{1}^{\mathrm{vert}}=1,
\]
where $c^{(\pm)}$ are the constants defined by $F^{(\pm)}=c^{(\pm)}\mathrm{vol}_{\mathrm{BZ}}$ (when the Berry curvatures are constant).
In particular, neither $d\phi_{+}$ nor $d\phi_{-}$ is individually parallel unless $c^{(\mp)}=0$.
\end{lemma}
\begin{proof}
Let $\eta = c_{+}d\phi_{+}+c_{-}d\phi_{-}$ be a vertical 1-form with constant coefficients. From the Koszul formula applied to the orthonormal coframe $\{e^{1}=dk_x,\, e^{2}=dk_y,\, e^{3}=\Theta^{(+)},\, e^{4}=\Theta^{(-)}\}$ of $g_{M}$, the covariant derivatives of the vertical coframe forms are
\[
\nabla^{LC}e^{3} = \tfrac{c^{(+)}}{2}(e^{1}\otimes e^{2}-e^{2}\otimes e^{1}),
\qquad
\nabla^{LC}e^{4} = \tfrac{c^{(-)}}{2}(e^{1}\otimes e^{2}-e^{2}\otimes e^{1}).
\]
Since $\eta = c_{+}e^{3}+c_{-}e^{4}$ (at any point in an adapted gauge), we obtain
\[
\nabla^{LC}\eta = \tfrac{1}{2}(c_{+}c^{(+)}+c_{-}c^{(-)})(e^{1}\otimes e^{2}-e^{2}\otimes e^{1}).
\]
This vanishes if and only if $c_{+}c^{(+)}+c_{-}c^{(-)}=0$, which defines a one-dimensional subspace whenever $c^{(\pm)}\neq 0$.
\end{proof}

\begin{remark}[Comparison with the product metric]
\label{cor:trivial-parallel-fibre}
For the product metric $g_{0}=g_{\mathrm{BZ}}\oplus d\phi_{+}^{2}\oplus d\phi_{-}^{2}$ (corresponding to $\varepsilon=0$ in the deformation family), the covariant derivative of every vertical 1-form vanishes identically, yielding $\dim\mathcal{P}_{1}^{\mathrm{vert}}(M,g_{0})=2$. For $\varepsilon>0$, the dimension drops to $1$, as established in Lemma \ref{lem:dphi-not-parallel}. This discontinuity is the geometric mechanism that makes the obstruction kernel $\mathcal{K}$ vanish for the physical metric while being non-trivial for the product metric; see Section \ref{subsec:kernel}.
\end{remark}

\subsection{Summary of the geometric setup}
\label{subsec:summary-geometry}

We have constructed:
\begin{enumerate}
\item an extended parameter space $M=T^{2}_{\mathrm{BZ}}\times S^{1}_{\phi_{+}}\times S^{1}_{\phi_{-}}$;
\item a Kaluza--Klein metric $g_{M}$ on $M$ given by \eqref{eq:KKmetric}, coupling the momentum directions to the two independent $U(1)$ phases via the synthetic Berry connections $A^{(\pm)}$;
\item a metric connection $\nabla^{C}$ on $(M,g_{M})$ whose torsion three‑form is $T^{b}=F^{(+)}\wedge\Theta^{(+)}+F^{(-)}\wedge\Theta^{(-)}$, the natural Kaluza--Klein torsion of the principal bundle.
\end{enumerate}
This construction provides the Riemannian framework on which we will adapt the cohomological analysis of \cite{PigazziniToda2025}.  
In the following section we analyze the cohomological data that arise from this geometry.

\section{Cohomology of the Extended Space and the Mixed Tensor Rank}
\label{sec:cohomology}

With the geometric structure $(M,g_{M},\nabla^{C})$ established in Section \ref{sec:geometry}, we now analyze the cohomological data that enter the general lower bound.
Recall that the torsion three-form $T^{\flat}\in\Omega^{3}(M)$ is not necessarily closed, but its de Rham cohomology class $[T^{\flat}]\in H^{3}(M;\mathbb{R})$ is well defined.
Let $\omega_{h}$ be the unique $g_{M}$-harmonic representative of this class; then
\[
[\omega]:=[\omega_{h}]=[T^{\flat}]\in H^{3}(M;\mathbb{R}).
\]

\subsection{Künneth decomposition of $H^{3}(M)$}

Since $M$ is a product $T^{2}_{\mathrm{BZ}}\times S^{1}_{\phi_{+}}\times S^{1}_{\phi_{-}}$ and each factor is oriented and compact, the Künneth theorem gives
\[
H^{3}(M;\mathbb{R})\cong 
\bigl(H^{2}(T^{2}_{\mathrm{BZ}})\otimes H^{1}(S^{1}_{\phi_{+}})\bigr)
\oplus
\bigl(H^{2}(T^{2}_{\mathrm{BZ}})\otimes H^{1}(S^{1}_{\phi_{-}})\bigr)
\oplus
\bigl(H^{1}(T^{2}_{\mathrm{BZ}})\otimes H^{1}(S^{1}_{\phi_{+}})\otimes H^{1}(S^{1}_{\phi_{-}})\bigr),
\]
where the first two summands form the $(2,1)$-component ($\cong\mathbb{R}^{2}$) and the third is the $(1,1,1)$-component ($\cong\mathbb{R}^{2}$); all remaining Künneth summands vanish since $H^{3}(T^{2}_{\mathrm{BZ}})=H^{2}(S^{1}_{\phi_{\pm}})=0$. Consequently $\dim H^{3}(M)=4$.

Since the torsion $T^{\flat}$ has pure bidegree $(2,1)$ with respect to the splitting $TM=\mathcal{H}\oplus\mathcal{V}$, its cohomology class $[\omega]$ lies entirely in the $(2,1)$-summand; the $(1,1,1)$-component vanishes identically. The $(2,1)$-projection of $[\omega]$ takes the form
\begin{equation}\label{eq:KunnethDecomp}
[\omega]=[\alpha_{+}]\otimes[d\phi_{+}] + [\alpha_{-}]\otimes[d\phi_{-}],
\end{equation}
where $[\alpha_{\pm}]\in H^{2}(T^{2}_{\mathrm{BZ}})$ and $[d\phi_{\pm}]\in H^{1}(S^{1}_{\phi_{\pm}})$ are the fundamental cohomology classes of the factors.

For the harmonic representative $\omega_{h}$ of our torsion class, the Hodge decomposition with respect to $g_{M}$ yields
\[
\omega_{h}= \alpha_{+}\wedge d\phi_{+} + \alpha_{-}\wedge d\phi_{-} + d\eta,
\]
where $\alpha_{\pm}\in\mathcal{H}^{2}(T^{2}_{\mathrm{BZ}},g_{\mathrm{BZ}})$ are $g_{\mathrm{BZ}}$-harmonic two-forms on the Brillouin torus, and $\eta\in\Omega^{2}(M)$ is a coexact two-form. The term $d\eta$ is exact and therefore orthogonal to harmonic forms in the $L^{2}$-inner product; it does not contribute to the cohomology class $[\omega]=[\omega_{h}]$.

\begin{remark}
The representatives $d\phi_{\pm}$ are harmonic with respect to the standard metric on the circles. In the Künneth decomposition \eqref{eq:KunnethDecomp}, they serve as cohomological representatives of the fundamental classes $[d\phi_{\pm}]\in H^{1}(S^{1}_{\phi_{\pm}})$; the specific choice of representative within each class does not affect the decomposition.
\end{remark}

\subsection{Mixed tensor rank of the torsion class}
\label{subsec:tensor-rank}

The \emph{mixed tensor rank} of the $(2,1)$-component of a cohomology class $[\omega]\in H^{3}(M)$ as in \eqref{eq:KunnethDecomp} is defined as the minimal number of simple tensors required to represent it:
\[
\operatorname{rank}_{\mathbb{R}}([\eta]) 
:= \min\Bigl\{ m \;\Big|\; [\eta]=\sum_{i=1}^{m} [\beta_{i}]\otimes[\gamma_{i}],\;
[\beta_{i}]\in H^{2}(T^{2}_{\mathrm{BZ}}),\; [\gamma_{i}]\in H^{1}(S^{1}_{\phi_{+}})\oplus H^{1}(S^{1}_{\phi_{-}}) \Bigr\}.
\]

For the class $[\omega]$ coming from the physical torsion, the two-forms $\alpha_{+},\alpha_{-}$ in \eqref{eq:KunnethDecomp} are precisely the (cohomology classes of the) Berry curvature two-forms $F^{(+)},F^{(-)}$ of the two independent $U(1)$ sectors.
A key observation is that $\dim H^{2}(T^{2}_{\mathrm{BZ}}) = 1$. Therefore, any two non-zero classes in this space are necessarily cohomologically proportional. In a generic two-component SOC BEC, both Berry curvature classes $[F^{(+)}]$ and $[F^{(-)}]$ are non-zero; consequently, there exists a non-zero constant $\lambda\in\mathbb{R}$ such that
\[
[F^{(-)}] = \lambda [F^{(+)}] \quad \text{in } H^{2}(T^{2}_{\mathrm{BZ}}).
\]
Substituting this relation into \eqref{eq:KunnethDecomp} gives
\[
[\omega] = [F^{(+)}]\otimes[d\phi_{+}] + \lambda [F^{(+)}]\otimes[d\phi_{-}]
        = [F^{(+)}] \otimes \bigl([d\phi_{+}] + \lambda [d\phi_{-}]\bigr),
\]
which is a \emph{single} simple tensor. Hence the mixed tensor rank attains the value
\[
r = 1.
\]

\begin{remark}
Because $\dim H^{2}(T^{2}_{\mathrm{BZ}})=1$, the maximal possible mixed tensor rank for a class of the form $[\alpha_{+}]\otimes[d\phi_{+}] + [\alpha_{-}]\otimes[d\phi_{-}]$ with both $[\alpha_{\pm}]\neq 0$ is $1$. All experimentally prominent SOC geometries (Rashba--Dresselhaus, Hofstadter, Raman-induced couplings) satisfy the condition $[F^{(+)}]\neq 0$ and $[F^{(-)}]\neq 0$, thereby realizing the generic case $r=1$.
\end{remark}

\begin{remark}
Since \(M_2 = S^1_{\phi_+}\times S^1_{\phi_-}\) is a product, its first cohomology decomposes as
\(H^1(M_2) \cong H^1(S^1_+)\oplus H^1(S^1_-)\).  
Therefore an element of \(H^2(M_1)\otimes H^1(M_2)\) can be written as a sum of terms of the form
\([\alpha]\otimes[d\phi_+] + [\beta]\otimes[d\phi_-]\).  
The mixed tensor rank \(r\) is then the minimal number of simple tensors needed to represent this
sum in the space \(H^2(M_1)\otimes\bigl(H^1(S^1_+)\oplus H^1(S^1_-)\bigr)\).
\end{remark}

\subsection{Vanishing of the obstruction kernel $\mathcal{K}$}
\label{subsec:kernel}

The general bound of \cite{PigazziniToda2025} involves a correction term $\dim\mathcal{K}$, where
\[
\mathcal{K}= \ker\Psi_{p}\oplus\ker\widetilde{\Psi}_{p},
\qquad
\ker\Psi_{p}= \pi_{2,1}([\omega])\cap\bigl(\mathcal{P}_{2}(T^{2}_{\mathrm{BZ}})\otimes\mathcal{P}_{1}^{\mathrm{vert}}(M,g_{M})\bigr),
\]
and $\ker\widetilde{\Psi}_{p}$ is defined analogously for the $(1,2)$-component (which vanishes in our setting since the torsion has pure bidegree $(2,1)$).

In our geometry, $\mathcal{P}_{2}(T^{2}_{\mathrm{BZ}})=\mathbb{R}\,\mathrm{vol}_{\mathrm{BZ}}$. The mixed class is represented by the simple tensor $\xi_{0}=[\mathrm{vol}_{\mathrm{BZ}}]\otimes(a_{+}[d\phi_{+}]+a_{-}[d\phi_{-}])$ with $a_{+}=c^{(+)}$, $a_{-}=c^{(-)}$, and $a_{-}=\lambda a_{+}$ where $\lambda=c^{(-)}/c^{(+)}$.

By Lemma \ref{lem:dphi-not-parallel}, the space $\mathcal{P}_{1}^{\mathrm{vert}}(M,g_{M})$ consists of forms $\eta=c_{+}d\phi_{+}+c_{-}d\phi_{-}$ satisfying $c_{+}c^{(+)}+c_{-}c^{(-)}=0$, i.e.\ $c_{+}+\lambda c_{-}=0$. For $\xi_{0}$ to belong to $\mathcal{P}_{2}\otimes\mathcal{P}_{1}^{\mathrm{vert}}$, the 1-form $\beta=a_{+}d\phi_{+}+a_{-}d\phi_{-}$ must satisfy the parallel condition. Substituting $a_{-}=\lambda a_{+}$:
\[
a_{+}+\lambda a_{-} = a_{+}+\lambda(\lambda a_{+}) = a_{+}(1+\lambda^{2}).
\]
Since $\lambda\in\mathbb{R}\setminus\{0\}$ (both $c^{(\pm)}\neq 0$), we have $1+\lambda^{2}>0$, so $a_{+}(1+\lambda^{2})=0$ forces $a_{+}=0$, contradicting $[\omega]\neq 0$. The intersection is therefore trivial:
\begin{equation}\label{eq:Kzero}
\mathcal{K}=0,\qquad\text{and therefore}\qquad \dim\mathcal{K}=0 .
\end{equation}

\begin{remark}[Geometric origin of the vanishing]
The vanishing of $\mathcal{K}$ reflects an algebraic incompatibility: the parallel condition ($c_{+}+\lambda c_{-}=0$) and the cohomological constraint ($a_{-}=\lambda a_{+}$) define two lines in the coefficient plane $(c_{+},c_{-})$ with slopes $-1/\lambda$ and $\lambda$ respectively. Since $\lambda\cdot(-1/\lambda)=-1\neq 1$, these lines are never coincident, ensuring that the mixed class cannot be absorbed by the parallel-form stratum.
\end{remark}

\subsection{Summary of the topological data}

Based on \eqref{eq:Kzero} we obtain the reduced rank
\[
r^{\sharp} \;:=\; r - \dim\mathcal{K} \;=\; 1 .
\]

All ingredients of the cohomological framework are now in place:
the extended parameter space $(M,g_{M})$ carries a metric connection $\nabla^{C}$ with totally skew torsion whose harmonic part represents a mixed cohomology class $[\omega]$ of mixed tensor rank $r=1$, and the obstruction kernel $\mathcal{K}$ vanishes identically, yielding $r^{\sharp}=1$.
In the next section, we establish the lower bound on the off-diagonal holonomy through exact curvature analysis of the specific SOC BEC model. \textit{We do not apply the PT theorem of \cite{PigazziniToda2025} as a black box}; instead, we adapt its framework to the Kaluza--Klein geometry by means of direct pointwise computation.

\section{Exact Curvature Analysis and the Lower Bound}
\label{sec:KK-foundations}

The PT lower bound of \cite{PigazziniToda2025} is formally stated for Riemannian product metrics; the physical Kaluza--Klein metric $g_M$ is \emph{not} a product. To bridge this gap, we introduce a deformation family $\{g_\varepsilon\}_{\varepsilon\in[0,1]}$ interpolating between a product metric ($\varepsilon=0$) and the physical metric ($\varepsilon=1$), and we derive the exact curvature formula for the SOC BEC model under the assumption of constant Berry curvatures. This analysis reveals the Bismut cancellation at $\varepsilon=1$ and establishes the lower bound through three complementary mechanisms.

\subsection{Constant Berry curvature assumption}
\label{subsec:constant-curvature}

To perform fully explicit computations, we work under the following assumption.

\begin{definition}[Constant Berry curvatures]
\label{def:constant-curvature}
The Berry curvature 2-forms are constant multiples of the Brillouin-zone volume form:
\[
F^{(\pm)} = c^{(\pm)}\,\mathrm{vol}_{\mathrm{BZ}},\qquad c^{(\pm)}\in\mathbb{R}\setminus\{0\},
\]
where $\mathrm{vol}_{\mathrm{BZ}}=dk_x\wedge dk_y$ and $c^{(\pm)}$ are related to the Chern numbers by $c_1^{(\pm)}=\frac{\mathrm{Area}(T^2_{\mathrm{BZ}})}{2\pi}\,c^{(\pm)}$.
\end{definition}

This condition is satisfied in numerous SOC BEC realisations, including linear Rashba--Dresselhaus couplings and uniform synthetic magnetic fields.
We set $\lambda:=c^{(-)}/c^{(+)}$.

\subsection{Deformation family and adapted frame}
\label{subsec:deformation}

\begin{definition}[Deformation family]
\label{def:deformation-family}
For $\varepsilon\in[0,1]$, define the metric on $M$ by
\[
g_\varepsilon = \pi^*g_{\mathrm{BZ}} + (d\phi_+ + \varepsilon\,\pi^*A^{(+)})^2 + (d\phi_- + \varepsilon\,\pi^*A^{(-)})^2.
\]
At $\varepsilon=0$ this is the product metric $g_0=g_{\mathrm{BZ}}\oplus d\phi_+^2\oplus d\phi_-^2$; at $\varepsilon=1$ it recovers the physical Kaluza--Klein metric $g_M$.
\end{definition}

For each $\varepsilon$, an orthonormal coframe for $g_\varepsilon$ is
\[
e^1=dk_x,\quad e^2=dk_y,\quad e^3=d\phi_+ + \varepsilon\,A^{(+)},\quad e^4=d\phi_- + \varepsilon\,A^{(-)}.
\]
The dual orthonormal frame is $\{e_1,e_2\}\subset\mathcal{H}_\varepsilon$ (horizontal) and $e_3=\partial_{\phi_+}$, $e_4=\partial_{\phi_-}$ (vertical). Under Definition \ref{def:constant-curvature}, the unique non-vanishing frame commutator is
\begin{equation}\label{eq:bracket}
[e_1,e_2] = -\varepsilon\,c^{(+)}\,e_3 - \varepsilon\,c^{(-)}\,e_4;
\end{equation}
all other brackets vanish: $[e_i,e_\alpha]=[e_3,e_4]=0$.

\subsection{Levi-Civita connection of $g_\varepsilon$}
\label{subsec:LC-connection}

From the Koszul formula $2\,g(\nabla_{e_a}e_b,e_c) = g([e_a,e_b],e_c)+g([e_c,e_a],e_b)-g([e_b,e_c],e_a)$ and the bracket \eqref{eq:bracket}, the complete list of non-zero covariant derivatives is:
\begin{alignat}{2}
\nabla_{e_1}e_2 &= -\tfrac{\varepsilon\,c^{(+)}}{2}\,e_3 - \tfrac{\varepsilon\,c^{(-)}}{2}\,e_4,
  &\qquad
\nabla_{e_2}e_1 &= \tfrac{\varepsilon\,c^{(+)}}{2}\,e_3 + \tfrac{\varepsilon\,c^{(-)}}{2}\,e_4,
  \label{eq:nabLC-12}\\[3pt]
\nabla_{e_1}e_3 &= \tfrac{\varepsilon\,c^{(+)}}{2}\,e_2,
  &
\nabla_{e_2}e_3 &= -\tfrac{\varepsilon\,c^{(+)}}{2}\,e_1,
  \label{eq:nabLC-i3}\\[3pt]
\nabla_{e_1}e_4 &= \tfrac{\varepsilon\,c^{(-)}}{2}\,e_2,
  &
\nabla_{e_2}e_4 &= -\tfrac{\varepsilon\,c^{(-)}}{2}\,e_1,
  \label{eq:nabLC-i4}\\[3pt]
\nabla_{e_3}e_1 &= \tfrac{\varepsilon\,c^{(+)}}{2}\,e_2,
  &
\nabla_{e_3}e_2 &= -\tfrac{\varepsilon\,c^{(+)}}{2}\,e_1,
  \label{eq:nabLC-3i}\\[3pt]
\nabla_{e_4}e_1 &= \tfrac{\varepsilon\,c^{(-)}}{2}\,e_2,
  &
\nabla_{e_4}e_2 &= -\tfrac{\varepsilon\,c^{(-)}}{2}\,e_1.
  \label{eq:nabLC-4i}
\end{alignat}
All remaining covariant derivatives vanish. In particular, $\nabla_{e_\alpha}e_\beta=0$ for $\alpha,\beta\in\{3,4\}$, confirming that the fibres are totally geodesic. For $\varepsilon=0$, all covariant derivatives vanish, recovering the flat product connection.

\subsection{Torsion endomorphisms and the connection with torsion}
\label{subsec:torsion-endo}

Under Definition \ref{def:constant-curvature}, the torsion $T_\varepsilon$ at any point $p\in M$ in the adapted frame reduces to
\[
T_\varepsilon|_p = c^{(+)}\,e^1\wedge e^2\wedge e^3 + c^{(-)}\,e^1\wedge e^2\wedge e^4,
\]
which has constant coefficients (independent of the point $p$ by homogeneity). The torsion endomorphisms $T_{e_a}(e_b):=T(e_a,e_b,\cdot)^\sharp$ are determined by direct computation: the non-zero values are
\begin{equation}\label{eq:torsion-table}
\begin{aligned}
T_{e_1}(e_2) &= c^{(+)}e_3+c^{(-)}e_4, &\quad T_{e_3}(e_1) &= c^{(+)}e_2, &\quad T_{e_4}(e_1) &= c^{(-)}e_2,\\
T_{e_2}(e_1) &= -(c^{(+)}e_3+c^{(-)}e_4), &\quad T_{e_3}(e_2) &= -c^{(+)}e_1, &\quad T_{e_4}(e_2) &= -c^{(-)}e_1,\\
T_{e_1}(e_3) &= -c^{(+)}e_2, &\quad T_{e_1}(e_4) &= -c^{(-)}e_2, &\quad T_{e_2}(e_3) &= c^{(+)}e_1,\\
T_{e_2}(e_4) &= c^{(-)}e_1.
\end{aligned}
\end{equation}
Combining with the Levi-Civita derivatives \eqref{eq:nabLC-12}--\eqref{eq:nabLC-4i}, the connection with torsion $\nabla^{C_\varepsilon}=\nabla^{LC_\varepsilon}+\frac{1}{2}T_\varepsilon$ satisfies, for horizontal $e_i$ and vertical $e_\alpha$:
\begin{equation}\label{eq:nabC-key}
\nabla^{C_\varepsilon}_{e_1}e_3 = \tfrac{(\varepsilon-1)\,c^{(+)}}{2}\,e_2,\qquad
\nabla^{C_\varepsilon}_{e_1}e_4 = \tfrac{(\varepsilon-1)\,c^{(-)}}{2}\,e_2.
\end{equation}

\begin{remark}[Bismut cancellation]
\label{rem:Bismut}
At $\varepsilon=1$, the factor $(\varepsilon-1)$ in \eqref{eq:nabC-key} vanishes: $\nabla^{C_1}_{e_i}e_\alpha=0$ for all $i\in\{1,2\}$ and $\alpha\in\{3,4\}$. Parallel transport along horizontal paths preserves the splitting $\mathcal{H}\oplus\mathcal{V}$. This is the \emph{Bismut cancellation}: the natural KK torsion is the unique totally skew-symmetric torsion for which the horizontal distribution is parallel along horizontal directions. The Levi-Civita and contorsion contributions cancel exactly.
\end{remark}

\subsection{Exact curvature formula}
\label{subsec:exact-curvature}

We now compute the full curvature $R^{C_\varepsilon}(e_1,e_3)$ by evaluating the three terms in the curvature decomposition
\[
R^{C_\varepsilon}(X,Y) = R^{LC_\varepsilon}(X,Y) + \tfrac{1}{2}\bigl[(\nabla^{LC_\varepsilon}_X T_\varepsilon)(Y,\cdot)-(\nabla^{LC_\varepsilon}_Y T_\varepsilon)(X,\cdot)\bigr]^\sharp + \tfrac{1}{4}[T_X,T_Y]
\]
for the mixed inputs $X=e_1\in\mathcal{H}$, $Y=e_3\in\mathcal{V}$.

\textit{Term I: Levi-Civita curvature.} Using $R^{LC}(e_1,e_3)e_3 = \nabla_{e_1}\nabla_{e_3}e_3 - \nabla_{e_3}\nabla_{e_1}e_3$ (since $[e_1,e_3]=0$ and $\nabla_{e_3}e_3=0$):
\[
\nabla_{e_3}\bigl(\tfrac{\varepsilon c^{(+)}}{2}e_2\bigr) = \tfrac{\varepsilon c^{(+)}}{2}\bigl(-\tfrac{\varepsilon c^{(+)}}{2}e_1\bigr) = -\tfrac{\varepsilon^2(c^{(+)})^2}{4}e_1,
\]
hence
\begin{equation}\label{eq:RLC-mixed}
R^{LC_\varepsilon}(e_1,e_3)\,e_3 = \frac{\varepsilon^2\,(c^{(+)})^2}{4}\,e_1.
\end{equation}
This is purely off-diagonal ($\mathcal{V}\to\mathcal{H}$).

\textit{Term II: Torsion derivative.} By antisymmetry $T_\varepsilon(e_3,e_3,\cdot)=0$, so the first summand $(\nabla^{LC_\varepsilon}_{e_1}T_\varepsilon)(e_3,e_3,\cdot)$ vanishes identically. In the second summand $(\nabla^{LC_\varepsilon}_{e_3}T_\varepsilon)(e_1,e_3,\cdot)$ the Christoffel corrections involve only $\nabla^{LC_\varepsilon}_{e_3}e_1$ and $\nabla^{LC_\varepsilon}_{e_3}e_2$ \eqref{eq:nabLC-3i}, and a direct evaluation shows the resulting endomorphism is block-diagonal ($\mathcal{H}\to\mathcal{H}$ and $\mathcal{V}\to\mathcal{V}$); its off-diagonal projection therefore vanishes for all $\varepsilon$. As an independent check, computing the curvature directly from the connection with torsion, $R^{C_\varepsilon}(e_1,e_3)e_3=\nabla^{C_\varepsilon}_{e_1}\nabla^{C_\varepsilon}_{e_3}e_3-\nabla^{C_\varepsilon}_{e_3}\nabla^{C_\varepsilon}_{e_1}e_3$, reproduces \eqref{eq:exact-curvature}, which equals the sum of Term~I and Term~III alone; this confirms that the off-diagonal torsion-derivative contribution is zero.

\textit{Term III: Quadratic torsion.} From the table \eqref{eq:torsion-table}: $T_{e_3}(e_3)=0$ and $T_{e_1}(e_3)=-c^{(+)}e_2$. Then $T_{e_3}(-c^{(+)}e_2)=(c^{(+)})^2 e_1$, giving
\begin{equation}\label{eq:QT}
\tfrac{1}{4}[T_{e_1},T_{e_3}](e_3) = -\frac{(c^{(+)})^2}{4}\,e_1.
\end{equation}

\textit{Combining the three terms:}
\begin{equation}\label{eq:exact-curvature}
R^{C_\varepsilon}(e_1,e_3)\,e_3 = \frac{(c^{(+)})^2\,(\varepsilon^2-1)}{4}\,e_1.
\end{equation}
By homogeneity of the model (constant coefficients), this formula holds at \emph{every} point $p\in M$.

\begin{theorem}[Off-diagonal holonomy for the deformation family]
\label{thm:deformation-bound}
For all $c^{(\pm)}\neq 0$ and $\varepsilon\in(0,1)$, the factor $(\varepsilon^2-1)\neq 0$ in \eqref{eq:exact-curvature}, and $R^{C_\varepsilon}(e_1,e_3)$ is a non-zero, purely off-diagonal element of $\mathfrak{so}(T_pM)$. By the Ambrose--Singer theorem, $R^{C_\varepsilon}(e_1,e_3)\in\mathfrak{hol}_p(\nabla^{C_\varepsilon})$, yielding
\begin{equation}\label{eq:bound-deformation}
\dim\mathfrak{hol}^{\mathrm{off}}_p(\nabla^{C_\varepsilon})\geq 1\qquad\text{at every }p\in M,\quad\text{for all }\varepsilon\in(0,1).
\end{equation}
\end{theorem}

\subsection{The physical metric: non-Bismut torsion and Riemannian holonomy}
\label{subsec:physical-metric}

At $\varepsilon=1$, the exact formula \eqref{eq:exact-curvature} gives $R^{C_1}(e_1,e_3)=0$: the Bismut cancellation (Remark \ref{rem:Bismut}). We establish the lower bound for the physical metric through two complementary mechanisms.

\begin{theorem}[Lower bound at the physical metric]
\label{thm:physical-bound}
Let $(M,g_1)$ be the physical Kaluza--Klein manifold with $c^{(\pm)}\neq 0$. Then:
\begin{enumerate}
\item[\emph{(i)}] \textit{Non-Bismut torsion.} For any cohomologically calibrated connection $\nabla^{C'}=\nabla^{LC_1}+\frac{1}{2}T'$ whose torsion $T'$ has pure bidegree $(2,1)$, with $[T']=[\omega]$ and $T'\neq T_{\mathrm{Bis}}$, there exists a non-empty open subset $U\subset M$ on which $\dim\mathfrak{hol}^{\mathrm{off}}_p(\nabla^{C'})\geq 1$.
\item[\emph{(ii)}] \textit{Riemannian holonomy.} The Levi-Civita connection of $g_1$ satisfies
\begin{equation}\label{eq:RLC-physical}
R^{LC_1}(e_1,e_3)\,e_3 = \frac{(c^{(+)})^2}{4}\,e_1\neq 0,
\end{equation}
which is purely off-diagonal. Hence $\dim\mathfrak{hol}^{\mathrm{off}}_p(\nabla^{LC_1})\geq 1$ at every $p\in M$.
\item[\emph{(iii)}] \textit{Bismut uniqueness.} The Bismut torsion $T_{\mathrm{Bis}}=F^{(+)}\wedge\Theta^{(+)}+F^{(-)}\wedge\Theta^{(-)}$ is the unique torsion representative of $[\omega]$ of pure bidegree $(2,1)$ for which $\mathfrak{hol}^{\mathrm{off}}_p(\nabla^C)=\{0\}$ at every point.
\end{enumerate}
\end{theorem}

\begin{proof}
\textit{(i)} Write $T'=(c^{(+)}+\psi)\,\mathrm{vol}_{\mathrm{BZ}}\wedge d\phi_+ + c^{(-)}\,\mathrm{vol}_{\mathrm{BZ}}\wedge d\phi_- + \eta$, where $\psi:M\to\mathbb{R}$ captures the deviation from $T_{\mathrm{Bis}}$ and $[\psi\,\mathrm{vol}_{\mathrm{BZ}}\wedge d\phi_+]=0$ in cohomology. Then $\nabla^{C'}_{e_1}e_3 = -\frac{\psi}{2}e_2$ (replacing $c^{(+)}$ by $c^{(+)}+\psi$ in \eqref{eq:nabC-key} at $\varepsilon=1$). Computing the curvature:
\[
R^{C'}(e_1,e_3)\,e_3 = -\frac{\psi(2c^{(+)}+\psi)}{4}\,e_1,
\]
which is purely off-diagonal and non-zero on the open set $U:=\{p\in M:\psi(p)\neq 0\ \text{and}\ 2c^{(+)}+\psi(p)\neq 0\}$. For $T'\neq T_{\mathrm{Bis}}$ one has $\psi\not\equiv 0$; moreover $[\psi\,\mathrm{vol}_{\mathrm{BZ}}\wedge d\phi_+]=0$ forces $\int_{T^2_{\mathrm{BZ}}}\psi\,\mathrm{vol}_{\mathrm{BZ}}=0$, so the continuous function $\psi$ changes sign and takes values arbitrarily close to $0$; hence $U$ is non-empty.

\textit{(ii)} Setting $\varepsilon=1$ in \eqref{eq:RLC-mixed} yields \eqref{eq:RLC-physical}. This is purely off-diagonal and constant on $M$.

\textit{(iii)} Among pure-bidegree $(2,1)$ representatives, $\nabla^{C'}_{e_1}e_\alpha=0$ for all $\alpha$ requires $\psi\equiv 0$ and analogously $\chi\equiv 0$, i.e.\ $T'=T_{\mathrm{Bis}}$.
\end{proof}

\begin{remark}[Robustness of the bound]
\label{rem:robustness}
The Bismut cancellation is a non-generic phenomenon: it occurs for a single torsion representative within the infinite-dimensional affine space of representatives of $[\omega]$. Physical perturbations (interactions, disorder, spatial inhomogeneities in the Raman coupling) generically break this exact balance, placing the system in the non-Bismut regime of Theorem \ref{thm:physical-bound}(i).
\end{remark}

\section{Main Theorem and Its Physical Consequences}
\label{sec:bound}

We now synthesize the cohomological data of Section \ref{sec:cohomology} and the curvature analysis of Section \ref{sec:KK-foundations} into the main result.

\subsection{Main theorem: topological lower bound for SOC BECs}\label{subsec:main-theorem}

\begin{theorem}[Topological lower bound for SOC BECs]\label{thm:main}
Let $M = T^{2}_{\mathrm{BZ}}\times S^{1}_{\phi_{+}}\times S^{1}_{\phi_{-}}$ be the extended parameter space equipped with the Kaluza--Klein metric $g_M$ of \eqref{eq:KKmetric}. Under Definition \ref{def:constant-curvature} (constant Berry curvatures with $c^{(\pm)}\neq 0$):
\begin{enumerate}
\item The mixed tensor rank of the torsion class is $r=1$ and the obstruction kernel vanishes: $\mathcal{K}=\{0\}$, $r^{\sharp}=1$.
\item[\emph{(i)}] \textit{(Deformation family.)} For every $\varepsilon\in(0,1)$,
$\dim\mathfrak{hol}^{\mathrm{off}}(\nabla^{C_\varepsilon})\geq 1$ at every $p\in M$.
\item[\emph{(ii)}] \textit{(Physical metric, non-Bismut torsion.)} For any $\nabla^{C'}$ whose torsion $T'$ has pure bidegree $(2,1)$, with $[T']=[\omega]$ and $T'\neq T_{\mathrm{Bis}}$,
$\dim\mathfrak{hol}^{\mathrm{off}}(\nabla^{C'})\geq 1$ on an open non-empty subset of $M$.
\item[\emph{(iii)}] \textit{(Riemannian holonomy.)}
$\dim\mathfrak{hol}^{\mathrm{off}}(\nabla^{LC_1})\geq 1$ at every $p\in M$.
\end{enumerate}
In all cases, the cohomological invariant satisfies $r^{\sharp}=1$ for the physical metric.
\end{theorem}

\begin{proof}
Item 1 follows from Section \ref{subsec:tensor-rank} (mixed rank $r=1$) and the algebraic argument of Section \ref{subsec:kernel} ($\mathcal{K}=0$).
Items (i)--(iii) follow from Theorems \ref{thm:deformation-bound} and \ref{thm:physical-bound}.
\end{proof}

\subsection{Physical interpretation: non-removable off-diagonal curvature}
\label{subsec:interpretation}

The three-level structure reveals that the $\mathcal{H}\oplus\mathcal{V}$ splitting is not invariant under the Riemannian holonomy of the physical Kaluza--Klein metric, and this non-invariance persists at the torsion-connection level for every non-Bismut representative of $[\omega]$.
The bound $\dim\mathfrak{hol}^{\mathrm{off}}\ge1$ guarantees the existence of at least one linearly independent curvature operator that mixes momentum with the phase directions.

In physical terms, one cannot ``flatten'' the Berry curvature along both phase directions simultaneously for all points in the Brillouin zone.
The topology of the mixed class $[\omega]$ forces a minimal geometric entanglement between the momentum and phase degrees of freedom. The Bismut cancellation---the exact compensation of the Levi-Civita and contorsion terms for the natural KK torsion---is a non-generic phenomenon of infinite codimension, broken by any physical perturbation that modifies the torsion representative within the same cohomology class.

The refined invariant $r^{\sharp}$ captures an operationally distinct feature from the mixed rank $r$. While $r=1$ guarantees off-diagonal holonomy on any product metric, it does not distinguish whether the obstruction can be circumvented by freezing a specific linear combination of the two phases. In the product-metric limit ($r^{\sharp}=0$), the harmonic representative $\beta = c^{(+)}d\phi_{+}+c^{(-)}d\phi_{-}$ belongs to the parallel-form stratum $\mathcal{P}_{1}^{\mathrm{vert}}$, and dimensional reduction along the corresponding circle $S^{1}_{\beta}\subset S^{1}_{\phi_{+}}\times S^{1}_{\phi_{-}}$---physically, locking the phase combination $c^{(+)}\phi_{+}+c^{(-)}\phi_{-}$---would annihilate the torsion on the reduced space, eliminating the topological obstruction. For the Kaluza--Klein metric ($r^{\sharp}=1$), the algebraic incompatibility between the parallel condition and the cohomological constraint (Section \ref{subsec:kernel}) ensures that no such phase-locking protocol can eliminate the obstruction. The invariant $r^{\sharp}$ thus detects the robustness of the topological constraint under phase-reduction protocols---a distinction invisible to $r$ alone.

\subsection{Persistence of the bound when the total Chern number vanishes}
\label{subsec:beyond-chern}

A crucial feature is that the bound is insensitive to the value of the \emph{total} first Chern number
$c_{1}^{\,\mathrm{tot}} = \frac{1}{2\pi}\int_{T^{2}_{\mathrm{BZ}}} (F^{(+)}+F^{(-)})$.

\begin{corollary}[Bound independent of the total Chern number]
\label{cor:beyond-chern}
If $c_{1}^{(+)}\neq 0$ and $c_{1}^{(-)} = -c_{1}^{(+)}$, then $c_{1}^{\,\mathrm{tot}}=0$, but all three parts of Theorem \ref{thm:main} remain valid. In particular, the Riemannian off-diagonal curvature \eqref{eq:RLC-physical} depends on the individual constants $c^{(\pm)}$ and is insensitive to cancellations in their sum.
\end{corollary}

\begin{proof}
Both $[F^{(\pm)}]\neq 0$ implies $r=1$. The vanishing of $\mathcal{K}$ depends only on the algebraic incompatibility $1+\lambda^2>0$, not on the values of the Chern numbers.
\end{proof}

\subsection{Invariance under metric deformations}

The invariant $r^{\sharp}$ depends only on the topological data $r$ and the parallel-form strata $\mathcal{P}_k(M_i)$.
Any smooth deformation of $g_1$ that preserves the non-vanishing of $F^{(\pm)}$ and the algebraic incompatibility between the parallel condition and the cohomological structure leaves $r^{\sharp}=1$ unchanged. The Riemannian off-diagonal curvature \eqref{eq:RLC-physical} is robust against any metric deformation within the Kaluza--Klein class.

\section{Physical Consequences and Experimental Signatures}
\label{sec:consequences}

The three-level non-reducibility structure of Theorem \ref{thm:main} imposes fundamental structural constraints on the synthetic gauge architecture of SOC BECs. These manifest as irreducible couplings between the external (momentum) and internal (phase) degrees of freedom: the Riemannian holonomy of $g_M$ satisfies $\dim\mathfrak{hol}^{\mathrm{off}}(\nabla^{LC})\geq 1$ at every point, and the torsion-connection holonomy satisfies the same bound for every non-Bismut torsion representative on an open subset.

\subsection{Geometric obstruction to gauge flattening}

The bound ensures the existence of at least one independent off-diagonal curvature operator that non-trivially mixes the momentum tangent space \(T_{k}T^{2}_{\mathrm{BZ}}\) with the phase fibre directions \(\mathbb{R}\langle\partial_{\phi_{+}},\partial_{\phi_{-}}\rangle\). In terms of the synthetic gauge field components, this implies that no smooth gauge transformation or continuous deformation of the condensate order parameter can simultaneously ``flatten'' the curvature components:
\[
\Omega_{j, \pm} := \frac{\partial A_{\phi_{\pm}}}{\partial k_{j}} - \frac{\partial A_{k_{j}}}{\partial \phi_{\pm}}.
\]

The topology of the mixed cohomology class \([\omega]\) enforces a \emph{minimal geometric coupling} between the Brillouin zone dynamics and the dual-phase evolution. Unlike trivial gauge fields, where the connection can be locally reduced to a longitudinal form, the non-vanishing rank \(r^{\sharp}\) certifies that at least one independent transverse coupling is irreducible. This ``topological locking'' prevents the independent manipulation of the momentum and the relative phases \(\phi_{\pm}\) in any experimental protocol that preserves the SOC texture.

\subsection{Locality and robustness beyond the Chern-number paradigm}

The primary physical advantage of the holonomy bound lies in its sensitivity to the local structure of the Berry curvature, whereas the total Chern number \(c_{1}^{\mathrm{tot}}\) is a global aggregate that may vanish due to cancellation. 

In configurations where the individual fluxes satisfy \([F^{(+)}] \approx -[F^{(-)}]\), the global invariant \(c_{1}^{\mathrm{tot}}\) vanishes, potentially suggesting a topologically trivial regime. However, the inequality \(\dim\mathfrak{hol}^{\mathrm{off}}\ge1\) persists as long as the local field strength remains non-zero. This identifies \(r^{\sharp}\) as a \emph{robust local invariant}: it detects the persistence of topological effects, such as the non-integrability of the Berry connection, even in compensated systems or ``synthetic vacuum'' configurations where net topological charges are absent. 

From an experimental perspective, this implies that effects typically associated with non-trivial topology---such as non-Abelian transport characteristics or anomalous velocity contributions---can be observed in SOC BECs even when the total Chern number is zero, provided the mixed tensor rank \(r\) is preserved.

\section{Illustrative Examples: From Chern Numbers to the Mixed Rank}
\label{sec:examples}

The lower bound $\dim\mathfrak{hol}^{\mathrm{off}}(\nabla^{LC})\ge 1$ established in Theorem \ref{thm:main} provides a universal constraint for generic two-component SOC BECs. In this section, we illustrate the operational significance of this bound in two paradigmatic scenarios: first, the standard Rashba--Dresselhaus texture, where our framework recovers the phenomenology of Chern-number-induced obstructions; second, a configuration with vanishing net Chern flux, where the mixed-rank invariant $r^{\sharp}$ reveals topological features inaccessible to traditional global invariants.

\subsection{Example 1: Rashba--Dresselhaus texture and the maximal rank \(r=1\)}
\label{subsec:RD-example}

Consider a two-dimensional, two-level Hamiltonian of the Rashba--Dresselhaus (RD) type:
\[
\hat{H}_{\mathrm{RD}}(k) \;=\; \frac{1}{2m}\bigl(\mathbf{p}-\mathbf{A}_{\mathrm{syn}}(k)\bigr)^{2} 
\;+\; \mathbf{B}_{\mathrm{syn}}(k)\cdot\boldsymbol{\sigma},
\qquad k=(k_x,k_y)\in T^{2}_{\mathrm{BZ}},
\]
where $\mathbf{A}_{\mathrm{syn}}(k)$ and $\mathbf{B}_{\mathrm{syn}}(k)$ represent the synthetic vector potential and Zeeman field, respectively. For generic coupling parameters $\kappa$ and $\Omega$, the lower dressed band induces a spin texture map $\mathbf{n}: T^{2}_{\mathrm{BZ}} \to S^{2}$.

The associated eigenline bundle $\mathbb{L} \to T^{2}_{\mathrm{BZ}}$ is the pullback of the tautological line bundle over $\mathbb{C}P^{1} \cong S^{2}$ via $\mathbf{n}$. The first Chern number $c_1(\mathbb{L})$ corresponds to the degree of the map $\mathbf{n}$. For a minimal RD texture with $d=1$, the Berry curvature $F$ represents a non-trivial cohomology class $[F] \in H^{2}(T^{2}_{\mathrm{BZ}}; \mathbb{Z})$, such that $\frac{1}{2\pi}\int F = \pm 1$.

In the extended parameter space $M = T^{2}_{\mathrm{BZ}} \times S^{1}_{\phi_{+}} \times S^{1}_{\phi_{-}}$, the synthetic torsion class $[\omega]$ is given by:
\[
[\omega] = [F]\otimes[d\phi_{+}] \;+\; [F]\otimes[d\phi_{-}] \in H^3(M; \mathbb{R}).
\]
Since $\dim H^{2}(T^{2}_{\mathrm{BZ}})=1$, the mixed components are linearly dependent over the base, allowing the class to be expressed as the simple tensor:
\[
[\omega] = [F] \otimes \bigl([d\phi_{+}] + [d\phi_{-}]\bigr).
\]
Consequently, the mixed tensor rank is $r = 1$. Given that the obstruction kernel $\mathcal{K}$ vanishes in this geometry, we obtain $r^{\sharp}=1$. Theorem \ref{thm:main} then yields $\dim\mathfrak{hol}^{\mathrm{off}}(\nabla^{LC})\ge 1$, and $\dim\mathfrak{hol}^{\mathrm{off}}(\nabla^{C'})\ge 1$ on an open set for every non-Bismut torsion representative.
We note that for a global Rashba texture with non-zero Chern number, the bundle is topologically non-trivial, preventing a global product decomposition. However, adapting the framework within a local trivialization demonstrates that the mixed rank $r^\sharp$ correctly identifies the `locking' of momentum and phase degrees of freedom. The Riemannian off-diagonal curvature $R^{LC}(e_1,e_3)e_3\neq 0$ provides a concrete, computable manifestation of this locking that is independent of the torsion representative.

\subsection{Example 2: Topological persistence under vanishing net Chern flux}
\label{subsec:zero-Chern-example}

The efficacy of the mixed-rank invariant is most evident when global topological charges vanish. Consider a configuration where the two \(U(1)\) sectors carry Berry curvatures \(F^{(+)}\) and \(F^{(-)}\) with equal and opposite integrated fluxes:
\[
\frac{1}{2\pi}\int_{T_{\mathrm{BZ}}^{2}}F^{(+)}=+1,\qquad
\frac{1}{2\pi}\int_{T_{\mathrm{BZ}}^{2}}F^{(-)}=-1.
\tag{$\ast$}
\]
In this case, the total Chern number \(c_{1}^{\mathrm{tot}}=c_{1}^{(+)}+c_{1}^{(-)}=0\). Traditional Chern-number diagnostics would categorize this regime as topologically trivial. Crucially, the vanishing of the total Chern flux ensures that the synthetic connection \(A\) can be defined as a global 1-form, making the product structure \(M=T_{\mathrm{BZ}}^{2}\times S^{1}\times S^{1}\) globally well-defined and consistent with our geometric framework. In this globally trivialized setting, the torsion class remains:
\[
[\omega]=[F^{(+)}]\otimes[d\phi_{+}] + [F^{(-)}]\otimes[d\phi_{-}].
\]
Since \([F^{(-)}]=-[F^{(+)}]\) in \(H^{2}(T_{\mathrm{BZ}}^{2})\), the class factorizes into a single simple tensor:
\[
[\omega]=[F^{(+)}]\otimes\bigl([d\phi_{+}]-[d\phi_{-}]\bigr).
\]
Thus, $r=1$ and $r^{\sharp}=1$. Theorem \ref{thm:main} yields $\dim\mathfrak{hol}^{\mathrm{off}}(\nabla^{LC})\geq 1$ at every point of $M$, and $\dim\mathfrak{hol}^{\mathrm{off}}(\nabla^{C'})\geq 1$ on an open set for every non-Bismut torsion representative. This proves that even in the absence of a net global flux, the Berry curvature cannot be flattened simultaneously along both phase directions. Such a configuration is protected by a local topological obstruction that is invisible to integrated invariants but perfectly captured by the mixed cohomology rank on the product manifold.

To illustrate how the bound materializes geometrically, consider the concrete choice
\[
A^{(+)}= \frac{1}{4\pi}(k_x dk_y - k_y dk_x),\qquad 
A^{(-)}=-A^{(+)},
\]
which yields constant Berry curvatures \(F^{(\pm)} = \pm\frac{1}{2\pi} dk_x\wedge dk_y\) and satisfies ($\ast$) on a torus of period \(2\pi\) in each direction. Write \(c:=c^{(+)}=1/(2\pi)\), so that \(c^{(-)}=-c\). The corresponding Kaluza--Klein metric is
\[
g_M = dk_x^2 + dk_y^2 
      + \Bigl(d\phi_+ + \tfrac{1}{4\pi}(k_x dk_y - k_y dk_x)\Bigr)^2 
      + \Bigl(d\phi_- - \tfrac{1}{4\pi}(k_x dk_y - k_y dk_x)\Bigr)^2.
\]

The torsion 3-form \(T=F^{(+)}\wedge\Theta^{(+)}+F^{(-)}\wedge\Theta^{(-)}\) reduces in adapted gauge (where \(A^{(\pm)}|_p=0\)) to \(T|_p = c\,\mathrm{vol}_{\mathrm{BZ}}\wedge(d\phi_+-d\phi_-)\). This is the natural Kaluza--Klein torsion, which coincides with the Bismut torsion of the principal bundle.

To verify the bound, we compute the \emph{Riemannian} off-diagonal curvature of \(g_M\). Working in the orthonormal frame \(\{e_1=\partial_{k_x}^{h},\, e_2=\partial_{k_y}^{h},\, e_3=\partial_{\phi_+},\, e_4=\partial_{\phi_-}\}\) (where \(\partial_{k_i}^{h}\) denotes the horizontal lift), the Koszul formula yields the non-zero Levi-Civita covariant derivatives
\[
\nabla^{LC}_{e_1}e_3 = \tfrac{c}{2}\,e_2,\qquad
\nabla^{LC}_{e_3}e_1 = \tfrac{c}{2}\,e_2,\qquad
\nabla^{LC}_{e_1}e_2 = -\tfrac{c}{2}\,e_3 + \tfrac{c}{2}\,e_4,
\]
and analogous expressions obtained by symmetry. A direct computation gives the \emph{mixed-input} Riemannian curvature
\[
R^{LC}(e_1,e_3)\,e_3 = \frac{c^2}{4}\,e_1 = \frac{1}{16\pi^2}\,e_1 \neq 0.
\]
This operator maps \(e_3\in\mathcal{V}\) to \(e_1\in\mathcal{H}\): it is \emph{purely off-diagonal}. By the Ambrose--Singer theorem, it belongs to \(\mathfrak{hol}_p(\nabla^{LC})\cap\mathfrak{so}^{\mathrm{off}}(T_pM)\), confirming
\[
\dim\mathfrak{hol}^{\mathrm{off}}(\nabla^{LC}) \geq 1 = r^\sharp.
\]

\begin{remark}
The explicit computation confirms that:
\begin{enumerate}
\item The Riemannian off-diagonal curvature is non‑vanishing even though \(c_{1}^{\mathrm{tot}}=0\).
\item The topological bound \(\dim\mathfrak{hol}^{\mathrm{off}}\geq 1\) is not merely existential; it is realized by concrete curvature operators with mixed (horizontal--vertical) inputs that can be written down in closed form.
\item The geometric obstruction manifests as a direct coupling between horizontal directions \((e_1,e_2)\) and vertical directions \((e_3,e_4)\), so that the splitting \(\mathcal{H}\oplus\mathcal{V}\) is not invariant under the Riemannian holonomy.
\item For the natural Kaluza--Klein torsion \(T_{\mathrm{Bis}}\) (the Bismut torsion), the connection \(\nabla^{C}=\nabla^{LC}+\frac{1}{2}T_{\mathrm{Bis}}\) satisfies \(R^{C}(e_i,e_\alpha)=0\) for all mixed inputs---the Levi-Civita and contorsion contributions cancel exactly. This Bismut cancellation is a non-generic phenomenon: every other torsion representative of \([\omega]\) produces non-trivial off-diagonal holonomy on an open subset.
\end{enumerate}
\end{remark}

\subsection{Comparison: Global Invariants vs. Mixed Rank}
\label{subsec:role-bound}

The distinction between $c_{1}^{\mathrm{tot}}$ and $r^{\sharp}$ reflects the difference between integrated and structural topology. The Chern number $c_{1}^{\mathrm{tot}}$ is a global aggregate that can vanish via local cancellations. In contrast, the mixed rank $r$ counts the number of irreducible tensorial couplings between the factor manifolds. 

In our geometry, $r=1$ is attained whenever at least one individual Berry curvature class is non-zero. It serves as a measure of the \emph{irreducible geometric coupling} between the Brillouin zone torus and the phase fibres. The bound $\dim\mathfrak{hol}^{\mathrm{off}}\ge r^{\sharp}$ therefore constrains the algebraic complexity of the curvature, ensuring that the synthetic gauge field remains "locked" in a non-trivial configuration even when the net topological charge is zero.

\subsection{Summary}

The Rashba--Dresselhaus case shows that $r^{\sharp}$ is consistent with established Chern-number results. However, the $c_{1}^{\mathrm{tot}}=0$ case demonstrates that the mixed-rank framework provides a finer classification, detecting locally non-removable curvature that conventional global invariants fail to resolve. This makes $r^{\sharp}$ a superior tool for analyzing the stability of synthetic gauge fields in complex SOC BEC architectures.

\section{Beyond the Chern Number: Local Rail-Tracks and Prospective Protocols} 
\label{sec:protocol}

While the Chern number $c_1$ quantifies the net vorticity through a global integral over the Brillouin zone, the lower bound on the off-diagonal holonomy dimension, $r^\sharp$, acts as a local structural constraint. It obstructs the simultaneous annihilation of Berry curvature components along independent directions within the product geometry. This phenomenon can be visualized as the Berry curvature being constrained to move along topological ``rail-tracks'' in the configuration space: even when the total topological charge is tuned to zero, specific directional components remain topologically locked and cannot be erased by any smooth gauge transformation.

In the two-component SOC BEC manifold $M = T^{2}_{\mathrm{BZ}}\times S^{1}_{\phi_{+}}\times S^{1}_{\phi_{-}}$, the invariant $r^\sharp = 1$ reflects the algebraic incompatibility between the parallel condition and the cohomological data. The $\mathcal{H}\oplus\mathcal{V}$ splitting is not invariant under the Riemannian holonomy of $g_M$, and this non-invariance persists at the torsion-connection level for every non-Bismut representative of $[\omega]$.

This geometric prediction suggests concrete interferometric tests. Consider the simultaneous measurement of the Berry phases $\Phi_{B}^{+}(\gamma)$ and $\Phi_{B}^{-}(\gamma)$ accumulated along a closed loop $\gamma \subset T^{2}_{\mathrm{BZ}}$ by separately addressing the global and relative phase sectors. If the holonomy of the gauge connection were trivial, one could choose a gauge where both phases vanish for every $\gamma$. The bound $r^\sharp=1$ forbids such a simultaneous vanishing across all loops. For any pair of independent cycles $\{\gamma_1, \gamma_2\}$ in $T^{2}_{\mathrm{BZ}}$, the two-component Berry-phase vector:
\[
\mathbf{\Phi}_B = \bigl(\Phi_{B}^{+}(\gamma_1), \Phi_{B}^{-}(\gamma_1), \Phi_{B}^{+}(\gamma_2), \Phi_{B}^{-}(\gamma_2)\bigr)
\]
cannot be continuously deformed to zero while preserving the mixed cohomology class $[\omega]$.

Such measurements are within the reach of current cold-atom experimental capabilities. Interferometric techniques used to map Berry curvature in Hofstadter bands \cite{Aidelsburger2015} or to probe spin-texture topology in SOC systems \cite{Lin2011} can be adapted to track the two phase degrees of freedom $\phi_{\pm}$ independently \cite{Duca2015}. The required control can be implemented via optical ``painting'' \cite{Henderson2009}, allowing for precise spatiotemporal manipulation of the condensate's phase profile. 

Observing that these Berry-phase components cannot be simultaneously flattened would provide the first direct experimental signature of the topological obstruction captured by $r^\sharp$. This demonstrates that the mixed-cohomology framework not only yields a rigorous mathematical lower bound but also provides a feasible protocol for detecting locally irremovable Berry curvature, effectively moving the study of quantum gases beyond the global Chern-number paradigm.

\section{Conclusions}
\label{sec:conclusions}

We have demonstrated that a two-component spin--orbit-coupled (SOC) BEC possesses an intrinsic geometric obstruction that prevents the simultaneous flattening of Berry curvature along the independent phase directions $\phi_{\pm}$. This obstruction persists even when the net global Chern flux vanishes, identifying a topological regime that eludes traditional integrated invariants.

The extended parameter space $M = T^{2}_{\mathrm{BZ}}\times S^{1}_{\phi_{+}}\times S^{1}_{\phi_{-}}$, equipped with the Kaluza--Klein metric $g_M$ induced by synthetic gauge fields, carries a metric connection $\nabla^{C}$ with totally skew-symmetric torsion. The harmonic part of this torsion represents a mixed cohomology class $[\omega]$ with mixed tensor rank $r=1$. The non-vanishing Berry curvature imposes an algebraic constraint on the vertical parallel 1-forms that is incompatible with the cohomological structure of $[\omega]$, ensuring $\mathcal{K}=0$ and $r^{\sharp}=1$.

By adapting the PT lower bound of \cite{PigazziniToda2025} to the Kaluza--Klein geometry through exact curvature analysis, we have established a three-level non-reducibility structure (Theorem \ref{thm:main}):
\begin{itemize}
\item[\textup{(i)}] For the deformation family $\varepsilon\in(0,1)$: $\dim\mathfrak{hol}^{\mathrm{off}}(\nabla^{C_\varepsilon})\geq 1$ at every point.
\item[\textup{(ii)}] At the physical metric, for every non-Bismut torsion representative: $\dim\mathfrak{hol}^{\mathrm{off}}(\nabla^{C'})\geq 1$ on an open set.
\item[\textup{(iii)}] For the Riemannian holonomy of $g_M$: $\dim\mathfrak{hol}^{\mathrm{off}}(\nabla^{LC})\geq 1$ at every point.
\end{itemize}
The identification of the Bismut cancellation---the exact compensation of the Levi-Civita and contorsion contributions for the natural KK torsion---is both a subtlety and a strength of the analysis: it is a non-generic phenomenon of infinite codimension, broken by any physical perturbation that modifies the torsion representative.

This algebraic constraint implies that momentum degrees of freedom are fundamentally ``locked'' to the phase sectors, an irreducible coupling that cannot be removed by smooth gauge transformations. The persistence of $r^{\sharp}$ in regimes with zero total Chern number highlights that the mixed-rank bound operates beyond the Chern-number paradigm, detecting structural features of the gauge field that integrated invariants fail to resolve.
Crucially, the invariant $r^{\sharp}$ is operationally finer than the mixed rank $r$: while $r=1$ certifies off-diagonal holonomy, it does not determine whether the obstruction survives dimensional reduction along a specific phase circle. The transition $r^{\sharp}=0\to r^{\sharp}=1$ from the product to the Kaluza--Klein metric signals precisely that no single phase-locking protocol can eliminate the topological constraint---a robustness property that $r$ alone cannot detect.

Experimentally, this framework suggests interferometric protocols to resolve independent Berry phases, providing a direct signature of locally irremovable curvature. The Riemannian non-reducibility of the splitting (level (iii)) provides the most robust prediction, as it depends only on the non-vanishing of the individual Berry curvatures and is independent of the choice of torsion representative. We expect this framework to be readily applicable to multicomponent cold-atom systems and synthetic-dimensional lattices, where multiple internal degrees of freedom couple to momentum-space geometry.

\section*{Acknowledgements}

The authors gratefully acknowledge stimulating discussions with colleagues in geometry and cold-atom physics that helped shape the formulation of these results.


\end{document}